\documentclass{amsart}
\usepackage{amssymb,latexsym}
\theoremstyle{plain}
\newtheorem{theorem}{Theorem}

\newtheorem{lemma}{Lemma}
\theoremstyle{definition}

\newtheorem{example}{Example}

\newtheorem{remark}{Remark}

\begin{document}

\title[symmetric tensor rank]
{Sets computing the symmetric tensor rank}
\author{Edoardo Ballico}
\address{Edoardo Ballico\\ 
 University of Trento\\ Department of Mathematics\\
I -- 38123 Povo (TN), Italy}
\email{ballico@science.unitn.it}
\thanks{The authors are partially supported by MIUR and GNSAGA of INdAM (Italy).}
\author{Luca Chiantini}
      \address{Luca Chiantini\\
Universita'  di Siena\\
     Dipartimento di Scienze Matematiche e Informatiche\\
     Pian dei Mantellini, 44\\
     I -- 53100 Siena, Italy}
   \email{chiantini@unisi.it}

\begin{abstract}
Let $\nu _d: \mathbb {P}^r \to \mathbb {P}^N$, $N:= \binom{r+d}{r}-1$, 
denote the degree $d$ Veronese embedding of $\mathbb {P}^r$.
For any $P\in \mathbb {P}^N$, the symmetric tensor rank $sr (P)$ is the minimal cardinality
of a set $S\subset \nu _d(\mathbb {P}^r)$ spanning $P$. Let
$\mathcal {S}(P)$ be the set of all $A\subset \mathbb {P}^r$ such that
$\nu _d(A)$ computes $sr (P)$. Here we classify all $P\in \mathbb {P}^n$
such that $sr (P) < 3d/2$ and $sr (P)$ is computed by at least two subsets 
of $\nu _d(\mathbb {P}^r)$. For such tensors $P\in \mathbb {P}^N$, 
we prove that $\mathcal {S}(P)$ has no isolated points.
\end{abstract}

\maketitle

\section{Introduction}\label{S1}

Let $\nu _d: \mathbb {P}^r \to \mathbb {P}^N$, $N:= \binom{r+d}{r}-1$, 
denote the degree $d$ Veronese embedding of $\mathbb {P}^r$.
Set $X_{r,d}:= \nu _d(\mathbb {P}^r)$.
For any $P\in \mathbb {P}^N$, the {\emph {symmetric rank}} or {\emph 
{symmetric tensor rank}} or, just, the {\emph {rank}} 
$sr (P)$ of $P$ is the minimal cardinality of a finite set 
$S\subset X_{r,d}$ such that $P\in \langle S\rangle$, where
$\langle \ \ \rangle$ denote the linear span.

For any $P\in \mathbb {P}^N$, let $\mathcal {S}(P)$ denote the set
of all finite subsets $A\subset \mathbb {P}^r$ such that $\nu _d(A)$
computes $sr(P)$, 
i.e. the set of all  $A\subset \mathbb {P}^r$ such that
$P\in \langle \nu _d(A)\rangle$ and $\sharp (A) = sr (P)$. Notice that  
 if $A\in \mathcal {S}(P)$, then
$P\notin \langle \nu _d(A')\rangle$ for any $A'\subsetneq A$.

The study of the sets $\mathcal {S}(P)$ has a natural role in the theory
of symmetric tensors. Indeed, if we interpret points 
$P\in \mathbb {P}^n$ as symmetric tensors, then $\mathcal {S}(P)$
is the set of all the representations of $P$ as a sum
of rank $1$ tensors. For many applications, it is crucial
to have some information about the structure of $\mathcal {S}(P)$.
We do not recall the impressive literature on the subject
(but see \cite{kb}, for a good references' repository).
The interest in the theory is growing, since applications of tensors are
actually increasing in Algebraic Statistics, and then in Biology, Chemistry
and  also Linguistics (see e.g. \cite{kb} and \cite{l}).
Let us mention one relevant aspect, from our point of view.
If we are looking for one specific decomposition of $P$ as a sum of tensors of
rank $1$, and we find some decomposition (there is a software, which tries 
heuristically to compute it), how to ensure that the found
decomposition is the expected one? Of course, if $\mathcal {S}(P)$
is a singleton, the answer is obvious. In a recent paper
(\cite{bgl}) Buczy\'{n}ski, Ginensky and Landsberg proved
that $\sharp (\mathcal {S}(P)) =1$ when the rank is small, i.e.
$sr (P) \le (d+1)/2$. This important uniqueness theorem 
(which holds more generally for $0$-dimensional schemes, see \cite{bl} Proposition 2.3) 
turns out to be sharp, even if $r=1$. For larger values of the rank, one can determine
the uniqueness of the decomposition, when
an element $A\in\mathcal{S}(P)$ satisfies some geometric properties 
(e.g. when no $3$ points of $A$ are collinear, see \cite{bb}, Theorem 2
or when $A$ is in {\it general uniform position}, see \cite{bc}).

In this paper, we describe more closely the set $\mathcal {S}(P)$, 
for tensors whose rank sits in the range  $sr(P)< 3/2$.

In particular, we show that for each $P$ with
$\sharp (\mathcal {S}(P)) >1$, the set $\mathcal {S}(P)$ has no 
isolated points. 

This result has a consequence. Assume we are given
$Q\in \mathbb {P}^n$ with $sr (Q) < 3d/2$, and we find 
$A\in \mathcal {S}(Q)$ which is isolated in $\mathcal {S}(Q)$. 
Then we can conclude that $A$ is the unique element of $\mathcal {S}(Q)$
(in other words, $Q$ is {\it identifiable}).
This means that, in the specified range, given
one decomposition $A\in \mathcal{S}(P)$, one can conclude that
$A$ is unique, just by performing an analysis
$\mathcal{S}(P)$ in a neighbourhood of $A$. This sounds to be
much easier than looking for other points of $\mathcal{S}(P)$
in the whole space.
 
Our precise statement is:
 
\begin{theorem}\label{i2}
Assume $r\ge 2$. Fix a positive integer $t<3d/2$. 
Fix $P\in \mathbb {P}^N$ such that $sr (P)=t$ and the 
symmetric rank of $P$ is computed by at least two different sets 
$A, B\subset \mathbb {P}^r$. Then
$sr (P)$ is computed by an infinite family of subsets of 
$\mathbb {P}^r$, and this family has no isolated points.
\end{theorem}

We notice that the notion of ``isolated points~'' requires an algebraic 
structure of the set $\mathcal {S}(P)$. As well-known (and checked 
in Section \ref{S2}), the set $\mathcal {S}(P)$ is constructible 
in the sense of Algebraic Geometry
(\cite{h}, Ex. II.3.18 and Ex. II.3.19). This makes more precise 
the expression ``~no isolated point~'' above (see Remark \ref{o000} 
in section \ref{S2} for the details).

We also prove that the bound $t<3d/2$, in the statement of Theorem \ref{i2}, is 
sharp. Indeed, Example \ref{i1} provides one tensors $P$ with $sr (P) =3d/2$
(so $d$ is even), and $\sharp (\mathcal {S}(P))=2$.

In the proof, it is not difficult to see that if there are
at least two elements in $\sharp (\mathcal {S}(P))=2$, when
$sr(P)<3d/2$, then the shape of the Hilbert functions of $A,B$ 
shows that both sets have a large intersection 
with either a line, or a conic of $\mathbb{P}^r$
(we will refer to \cite{bb} and \cite{c2}, for this part of the theory).
Then, we perform a (maybe tedious, but necessary) analysis
of the behaviour of sets of points, with a big intersection
with either a line or  a conic.

We also provide a deeper description of $\mathcal {S}(P)$,
still in the range $sr(P)<3/2$ and   
assuming that $\mathcal {S}(P)$ is not a singleton (hence it is infinite). 
Indeed, we have the following:

\begin{theorem}\label{i3}
Assume $r\ge 2$ and $d\ge 3$. Fix a positive integer $t<3d/2$. 
Fix $P\in \mathbb {P}^N$ such that $sr (P)=t$. Then, the set $\mathcal
{S}(P)$ is not a single point if and only if $P$ may be described 
in one of the following way:
\begin{itemize}
\item[(a)] for any $A\in \mathcal {S}(P)$, there is a line 
$D\subset \mathbb {P}^r$ such that $\sharp (A\cap D) \ge \lceil 
(d+2)/2\rceil$; set $F:= A\setminus A\cap D$; the set
$\langle \nu _d(A\cap D)\rangle \cap \langle\{P\}\cup \nu _d(F)\rangle$, 
is formed by a unique point $P_D$ and $\mathcal {S}(P_D)$ is infinite; 
for each $E\in \mathcal {S}(P_D)$ we have $E\cap F =\emptyset$
and $E\cup F \in \mathcal {S}(P)$.
\item[(b)]  for any $A\in \mathcal {S}(P)$, there is a smooth conic 
$T\subset \mathbb {P}^m$ such that $\sharp (A\cap T) \ge d+1$; 
set $F:= A\setminus A\cap T$; the set
$\langle \nu _d(A\cap T)\rangle \cap \langle \{P\}\cup F\rangle$, 
is formed by a unique point $P_T$ and $\mathcal {S}(P_T)$ is infinite; 
for each $E\in \mathcal {S}(P_T)$ we have $E\cap F =\emptyset$;
every element of $\mathcal {S}(P)$ is of the form $E'\cup F$ for some 
$E'\subset T$ computing $\mathcal {S}(P_T)$ with respect to the 
rational normal curve $\nu _d(T)$.
\item[(c)] $d$ is odd; for any $A\in \mathcal {S}(P)$, there is 
a reducible conic $T=L_1\cup L_2\subset \mathbb {P}^m$, $L_1\ne
L_2$, such that $\sharp (A\cap L_1) = \sharp (A\cap L_2)=(d+1)/2$ and 
$L_1\cap L_2\notin A$.
\end{itemize}
\end{theorem}

Let us mention that if $L$ is a linear subspace of 
dimension $m$ in $\mathbb{P}^r$,
then the Veronese embedding $\nu_d$, restricted to
$L$, can be identified with a $d$-th Veronese embedding
of $\mathbb{P}^s$. Thus, if $Q$ is a point of the linear
span $\langle \nu_d(L)\rangle$, then we can consider
the rank  of $Q$, either with respect to $X_{r,d}$, or with respect
to $X_{m,d}$. Fortunately, in our cases where 
this ambiguity could arise, \cite{bl} Corollary 2.2
will guarantee that the two ranks are equal, and 
every decomposition $A\in\mathcal {S}(Q)$, with respect to
$X_{r,d}$, is contained in $X_{m,d}$. Indeed, we have:

\begin{remark}\label{i4}
Take $P_D$ (resp. $P_T$) as in case (a) (resp. (b)) of 
Theorem \ref{i3}. By \cite{ls}, Proposition 3.1, or \cite{lt}, 
subsection 3.2, $sr (P_D)$ (resp. $sr (P_T)$) is equal to its 
symmetric rank with respect to the rational normal curve 
$\nu _d(D)$ (resp. $\nu _d(T))$. By the symmetric case of 
\cite{bl}, Corollary 2.2, each element of $\mathcal {S}(P_D)$ 
(resp. $\mathcal {S}(P_T)$) is contained in $D$ (resp. $T$). 

Several algorithms are available, to get an element of $\mathcal 
{S}(P_D)$ or $\mathcal {S}(P_T)$ (\cite{cs}, \cite{lt}, \cite{bgi}).
\end{remark}

Finally, we wish to thank J. Landsberg, who pointed out to us
the importance of studying the existence of isolated points $A\in \mathcal 
{S}(P)$, when $\mathcal {S}(P)$ is not a singleton.

\section{Preliminaries}\label{S2}

We work over an algebraically closed field $\mathbb {K}$ 
such that $\mbox{char}(\mathbb {K})=0$.
\smallskip

Recall, from the introduction, than $\nu _d: \mathbb {P}^r \to \mathbb {P}^N$, 
$N:= \binom{r+d}{r}-1$ denotes the degree $d$ Veronese embedding of $\mathbb {P}^r$.
Call $X_{r,d}$ the image of this map.

For any closed subscheme $W\subseteq \mathbb {P}^r$, let $\langle W\rangle$
denote the linear span of $W$. If $W$ sits in some hyperplane, $\langle W\rangle$
is  the intersection of all the hyperplanes of $\mathbb {P}^r$ containing $W$. 

For any integer $m>0$ and any integral, positive-dimensional subvariety 
$T\subset \mathbb {P}^r$, we  let $\Sigma _m(T)$ denote the embedded $m$-th
secant variety of $X$, i.e. the closure in $\mathbb {P}^r$ 
of the union of all $(m-1)$-dimensional linear subspaces spanned by $m$ points of $T$.
We take the closure with respect to the Zariski topology. Notice that, over 
the complex number field, the closure in the euclidean topology gives the same set.

For any integer $k>0$, let $\mbox{Hilb}^k(\mathbb {P}^r)^0$ denote the set of 
all finite ($0$--dimensional) reduced subsets of $\mathbb {P}^r$, with cardinality $k$.  
$\mbox{Hilb}^k(\mathbb {P}^r)^0$ is a smooth and quasi-projective variety of dimension $rk$.

\begin{remark}\label{o000}
We observe that the set $\mathcal {S}(P)$, defined in the
introduction,  is always constructible.

Indeed, let $G:=G(k-1,r)$ denote the Grassmannian of all $(k-1)$-dimensional 
linear subspaces of $\mathbb {P}^r$. 
For any point $P\in \mathbb {P}^r$, set $G(k-1,r)(P):= \{V\in G(k-1,r): P\in V\}$ 
and $G(k-1,r)(P)_+:= \{V\in G(k-1,r)(P):P$ is spanned by $k$ points of $V\cap X$. 
Notice that, by definition, $G(k-1,r)(P)_+ =\emptyset$ for all $k < sr (P)$
and $G(sr (P)-1,r)(P)_+ \ne \emptyset$. 
Now, put $\mathcal {J}:= \{(S,V)\in  \mbox{Hilb}^{sr (P)}(\mathbb {P}^r)^0\times G(sr (P)-1,k)(P)_+ : 
P\in \langle \nu _d(S)\rangle \}$. This set $\mathcal {J}$
is locally closed. If $\pi_1$ denotes the projection onto the first factor, then
$\mathcal {S}(P)$ is exactly the image $\pi _1(\mathcal {J})$. 
Hence, a theorem of Chevalley guarantees that $\mathcal {S}(P)$ is a constructible set 
(\cite{h}, Ex. II.3.18 and Ex. II.3.19). 

We are interested in isolated points of $\mathcal {S}(P)$. Notice that $Z$ is an
isolated point for $\mathcal {S}(P)$ when $Z$ is an irreducible component of the closure of
$\mathcal {S}(P)$. Thus, the notion of {\it isolated points} for $\mathcal {S}(P)$
are equal both if we use the Zariski or the Euclidean topology on $\mathcal {S}(P)$.
\end{remark}

\begin{remark}\label{a00}
Let $X$ be any projective scheme and $D$ any effective Cartier divisor of $X$. For any closed 
subscheme $Z$ of $X$, we denote with $\mbox{Res}_D(Z)$ the residual scheme of $Z$ with 
respect to $D$. i.e. the closed subscheme of $X$ with ideal sheaf 
$\mathcal {I}_Z:\mathcal {I}_D$ (where $\mathcal {I}_Z, \mathcal {I}_D$ are the ideal sheaves
of $Z$ and $D$, respectively). 

We have $\deg (Z) = \deg (Z\cap D) + \deg (\mbox{Res}_D(Z))$. If $Z$ is a finite reduced set, 
then $\mbox{Res}_D(Z) = Z\setminus Z\cap D$. For every $L\in \mbox{Pic}(X)$ we have the exact sequence
\begin{equation}\label{eqa1.0}
0 \to \mathcal {I}_{\mbox{Res}_D(Z)}\otimes L(-D) \to \mathcal {I}_Z\otimes L \to 
\mathcal {I}_{Z\cap D,D}\otimes (L\vert D)
\to 0
\end{equation}
From (\ref{eqa1.0}) we get
$$h^i(X,\mathcal {I}_Z\otimes L) \le h^i(X,\mathcal {I}_{\mbox{Res}_D(Z)}
\otimes L(-D))+h^i(D,\mathcal {I}_{Z\cap
D,D}\otimes (L\vert D))$$ for every integer $i\ge 0$.
\end{remark}

\section{The proofs}\label{S3}

We will make an extensive use of the following two results.

\begin{lemma}\label{v1}  Let $A,B\in \mathbb{P}^r$ 
be two zero-dimensional schemes such that $A\ne B$. 
Assume the existence of $P\in \langle \nu _d(A)\rangle \cap \langle
\nu _d(B)\rangle$ such that $P\notin \langle \nu _d(A')\rangle$ 
for any $A'\subsetneq A$ and $P\notin \langle \nu _d(B')\rangle$ for any $B'\subsetneq B$.
Then $h^1(\mathbb P^r,\mathcal {I}_{A\cup B}(d)) >0$.
\end{lemma}
\begin{proof} See \cite{bb}, Lemma 1.
\end{proof}

The following lemma was proved (with $D$ a hyperplane) 
in \cite{bb1}, Lemma 7. The same proof works for an arbitrary 
hypersurface $D$ of $\mathbb {P}^r$.

\begin{lemma}\label{v2}
Fix positive integers $r, d, t$ such that $t\le d$ and finite sets 
$A, B\subset \mathbb {P}^r$. Assume the existence of a degree $t$ hypersurface 
$D\subset \mathbb {P}^r$ such that $h^1(\mathcal {I}_{(A\cup B)\setminus 
(A\cup B)\cap D}(d-t)) =0$. Set $F:= A\cap B \setminus (D\cap A\cap B)$. 

Then $\nu _d(F)$ is linearly independent. Moreover $\langle \nu _d(A)\rangle
\cap \langle \nu _d(B)\rangle$ is the linear span of the two supplementary subspaces 
$\langle \nu _d(F)\rangle$ and  $\langle \nu _d(A)\rangle
\cap \langle \nu _d(B)\rangle$. 

Assume there is $P\in \langle \nu _d(A)\rangle \cap \langle
\nu _d(B)\rangle$ such that $P\notin \langle \nu _d(A')\rangle$ for any $A'\subsetneq A$, 
and $P\notin \langle \nu _d(B')\rangle$ for any $B'\subsetneq B$. Then
$A = (A\cap D)\sqcup F$, $B = (B\cap D)\sqcup F$ and $A\setminus A\cap D = B\setminus B\cap D$.
\end{lemma}

Next, we need to point out first the case
of the Veronese embeddings  $X_{1,d}$ of $\mathbb P^1$.
This (already non--trivial) case anticipates some features
 of the behaviour of the sets $\mathcal S(P)$, in higher dimension.

\begin{lemma}\label{w1}
Assume $r=1$ and hence $N =d$. Fix $P\in \mathbb {P}^d$ such that $sr (P)$ is 
computed by at least two different subsets of $X_{1,d}$. 
Then $\dim (\mathcal {S}(P)) > 0$ and $\mathcal {S}(P)$ has no isolated points.
\end{lemma}

\begin{proof}
Let $t$ be the border rank of $P$, i.e. the minimal integer  such that $P$ sits in
the secant variety $\Sigma _t(X_{1,d})$. The dimension of secant varieties of
irreducible curve is well known (\cite{a}, Remark 1.6), and it turns out that 
$t \le \lfloor (d+2)/2\rfloor$. Take $A, B$ computing $sr (P)$ and such that 
$A\ne B$. Lemma \ref{v1} gives $h^1(\mathcal {I}_{A\cup B}(d)) >0$. 
Since any set of at most $d+1$ points is separated by divisors of degree $d$, we 
see that $\sharp (A\cup B) \ge d+2$. Hence $\sharp (A) =\sharp (B)\ge t$
and equality holds only if $t =(d+2)/2$ and $A\cap B=\emptyset$.

\quad (i) First assume $t =(d+2)/2$, so that, as we observed above, $t$
is also the symmetric rank of $P$.
In this case, by \cite{a}, Remark 1.6, a standard dimensional count
proves that $\Sigma _{t}(X_{1,d}) = \mathbb {P}^d$.
Moreover, $(\mathcal {S}(Q))$ can be described as the fiber of a natural
proper map of varieties. Namely, let $G(t-1,d)$ denotes the Grassmannian 
of $(t-1)$-dimensional linear subspaces of $\mathbb {P}^d$. 
Let $\mathbb {I} := \{(O,V)\in \mathbb {P}^d\times G(t-1,d): O\in V\}$ 
denote the incidence correspondence, and $\pi _1$, $\pi _2$ denote the morphisms induced
from the projections to the two factors. 
Since $X_{1,d}$ is a rational normal curve, of degree $d$, notice that
$\dim (\langle W\rangle )=t-1$ for {\it every} $W\in \mbox{Hilb}^t(X_{1,d})$. 
Thus, the  map $Z\mapsto \langle Z\rangle$ defines
a proper morphism $\phi :  \mbox{Hilb}^t(X_{1,d})\to G(t-1,d)$. 
Set $\Phi := \pi _2^{-1}(\phi (\mbox{Hilb}^t(X_{1,d})))$. 
By construction, $\mathcal {S}(P)$ corresponds to the fiber of the map 
$\pi_{1|Phi} : \Phi \to \mathbb {P}^d$ over $P$. 
$\Phi$ (the {\it abstract secant variety}) is an integral variety of
dimension $\dim (\Phi )=d+1$ (\cite{a}). 
Since $\psi$ is proper and $\Phi$ is integral, every fiber of $\pi_{1|Phi}$
has dimension at least $1$ and no isolated points (\cite{h}, Ex. II.3.22 (d)). 
Thus, the claim holds, in this case.

\quad (ii) Now assume $d \ge 2t-1$. Hence $t< sr (P)$. A theorem of Sylvester 
(see \cite{cs}, or \cite{lt}, Theorem 4.1) proves that, in this case, 
$sr (P) = d+2 -t$. Moreover, by \cite{lt} \S 4, there is a unique 
zero-dimensional scheme $Z\subset \mathbb {P}^1$ such that $\deg (Z)=t$ and 
$P\in \langle \nu _d(Z)\rangle$. As $t<sr(P)$, this subscheme $Z$ cannot
be reduced.
 
Fix any $A\in \mathcal {S}(P)$. Since $h^1(\mathcal {I}_{A\cup Z}(d))>0$ (Lemma \ref{v1})
and $\deg (A)+\deg (Z) =d+2$, we have $Z\cap A = \emptyset$. Fix any $E\subset A$ such that  
$d-\sharp (E) = 2t-2$. Let $Y_E\subset \mathbb {P}^{2t-2}$ be the image of $X_{1,d}$ 
under the projection $\pi _E$ from the linear subspace $\langle \nu _d(E)\rangle$. 
Notice that $Y_E$ is again a rational normal curve, of degree $2t-2$, so that it coincides, 
up to a projectivity, with $X_{1,2t-2}$. 
 
We have $Z\cap E = \emptyset$. Moreover $\deg (Z)+\sharp (E) \le d+1$, so that, by the
properties of the rational normal curve mentioned above, the set $\nu_d(Z)\cup\nu_d(E)$
is linearly independent. It follows $\langle \nu _d(Z)\rangle\cap \langle \nu _d(E)\rangle = \emptyset$.
Hence $\pi _E$ is a morphism  at each point of $\langle \nu _d(Z)\rangle$ and maps it isomorphically 
onto a $(t-1)$-dimensional linear subspace of $\mathbb {P}^{2t-2}$.  As $\deg (A) \le d+1$,
for the same reason we also have $\langle \nu _d(A\setminus E)\rangle \cap \langle \nu _d(E)\rangle 
= \emptyset$.  It follows that  the symmetric rank of $\pi_E(P)$ (with respect to $Y_E$) 
is exactly $t$, and $\pi _E(\nu _d(A\setminus E))$ is one of the elements
of the set $\mathcal {S}(\pi _E(P))$. Moreover, for any $U\in \mathcal {S}(\pi _E(P))$ 
the set $U\cup E$ computes $sr (P)$. We saw above that  $\pi _E(\nu _d(A\setminus E))$ 
is not an isolated  element of $\mathcal {S}(\pi _E(P))$. Thus $A$ is not an isolated 
element of $\mathcal {S}(P)$.
\end{proof}

\vspace{0.3cm}

Now, we are ready to prove our first main result.
\smallskip

\qquad {\emph {Proof of Theorem \ref{i2}.}} 
Since $A\ne B$, Lemma \ref{v1}  gives $h^1(\mathcal {I}_{A\cup B}(d)) >0$. 
Then, since $\sharp (A\cup B) \le 2t < 3d$, 
one of the following cases occurs (\cite{c2}, Th. 3.8):
\begin{itemize}
\item[(i)] there is a line $D\subset \mathbb {P}^r$ such that $\sharp (D\cap (A\cup B))\ge d+2$;
\item[(ii)] there is a conic $T\subset \mathbb {P}^r$ such that $\sharp (T\cap (A\cup B))\ge 2d+2$.
\end{itemize}
We will proof the statement, by showing that Lemma \ref{w1} implies that  
we can move the points of $A\cap D$ (in case (i)), or $A\cap T$ (in case (ii)), 
in a continuous family, whose elements, together with $A\setminus (A\cap D)$,
determine a non trivial family of sets in $\mathcal {S}(P)$, which generalizes $A$.    

\quad (a) In this step, we assume the existence of a line $D\subset \mathbb {P}^r$ such that 
$\sharp (D\cap (A\cup B))\ge d+2$. 

Set $F:= A\setminus (A\cap D)$. Let $H\subset \mathbb {P}^r$ be a general hyperplane containing $D$. 
Since $A\cup B$ is finite and $H$ is general, we have have $(A\cup B)\cap H = (A\cup B)\cap D$. 

First assume  $h^1(\mathcal {I}_{(A\cup B)\setminus (A\cup B)\cap D}(d-1))=0$.
Lemma \ref{v2} gives $A\setminus (A\cap D) = B\setminus (B\cap D)$. Hence $\sharp (A\cap D) = 
\sharp (B\cap D)$ and $A\cap D \ne B\cap D$, since $A\ne B$.
The Grassmann's formula shows that $\langle \nu _d(A)\rangle \cap 
\langle \nu _d(B)\rangle$ is the linear span of its (supplementary) subspaces 
$\langle \nu _d(A\setminus (A\cap D))\rangle$ and $\langle \nu _d(A\cap D)\rangle \cap 
\langle \nu _d(B\cap D)\rangle$. This means that one can find a point $P_D\in \langle \nu _d(A\cap D)\rangle 
\cap \langle \nu _d(B\cap D)\rangle$ such that $P\in \langle \{P_D\}\cup 
\nu _d(A\setminus A\cap D)\rangle = \langle \{P_D\}\cup 
\nu _d(F)\rangle$. We notice that $A\cap D$ and $B\cap D$ are two different
subsets of the rational normal curve $\nu_d(D)$, and they computes the rank of $P_D$, 
with respect to $\nu_d(D)=X$ (which can be identified with $X_{1,d}$, see
the Introduction). Indeed, if $P_D$ belongs to the span of a subset $Z$
of $\nu_d(D)$, with cardinality smaller than $A\cap D$, then $P$ would belong
to the span of the subset $\nu_d(F)\cup Z$, of cardinality smaller than $sr(P)$, a
contradiction. By Lemma \ref{w1}, $A\cap D$ is not an isolated point of $\mathcal{S}(P_D)$. 

\quad {\emph {Claim 1:}} Fix any $E\in \mathcal {S}(P_D)$. Then 
$sr (P) = \sharp (F) +sr (P_D)$ 
and $E\cup F\in \mathcal {S}(P)$.

\quad {\emph {Proof of Claim 1:}} Notice that, by the symmetric case of \cite{bl}, 
Corollary 2.2 (see also Remark \ref{i4}), 
every  element of $\mathcal {S}(P_D)$ is contained in $D$ and 
in particular it is disjoint from $F$. 
Since $P_D\in \langle \nu _d(E)\rangle$ and  $P\in \langle \{P_D\}\cup \nu _d(F)$, 
we have $P\in \langle \nu_d(E\cup F)\rangle $. 
Hence, to prove Claim 1 it is sufficient to prove $\sharp (E\cup F) \le sr (P)$. 
Since $F\cap D = \emptyset$, we have $\sharp (E\cup F) = sr (P) +sr (P_D) -\sharp (A\cap D)$. 
Since $P_D\in \langle \nu _d(A\cap D)\rangle$, we have $\sharp (A\cap D) \ge sr (P_D)$ 
by the definition of $sr (P_D)$, concluding the proof of Claim 1.

Claim 1 implies that $A$ is not an isolated 
point of $\mathcal {S}(P)$. Namely, let $\Delta$ be an integral affine curve and $o\in \Delta$
such that there is $\{\alpha _\lambda \}_{\lambda \in \Delta} \subseteq \mathcal {S}(P_D)$ 
with $\alpha _o = A\cap D$ and $\alpha _\lambda \subset D$ for all $\lambda \in \Delta$ 
(Lemma \ref{w1}). By Claim 1, we have $F\cup \alpha _\lambda \in \mathcal {S}(P)$ for
all $\lambda \in \Delta$. 

Now assume $h^1(\mathcal {I}_{(A\cup B)\setminus (A\cup B)\cap D}(d-1))>0$. Since 
$\sharp ((A\cup B)\setminus (A\cup B)\cap D) \le 2d-2\le 2d-1$, again there is a line
$L\subset \mathbb {P}^m$ such that $\sharp (L\cap ((A\cup B)\setminus (A\cup B)\cap D)) \ge d+1$. 
Let $H_2\subset \mathbb {P}^m$ be a general quadric hypersurface containing $D\cup L$ 
(it exists, because if $L\cap D =\emptyset$, then $r\ge 3$). Since $L\cup D$ is the base locus
of the linear system $\vert \mathcal {I}_{L\cup D}(2)\vert$, $A\cup B$ is finite and 
$H_2$ is general in  $\vert \mathcal {I}_{L\cup D}(2)\vert$,
we have $H_2\cap (A\cup B) = (L\cup D)\cap (A\cup B)$. By Lemma \ref{v2}, $A\setminus (A\cap 
(D\cup L)) = B\setminus (B\cap (D\cup L))$. Since $\sharp ((A\cup B)\setminus (A\cup B)\cap H_2) 
\le 3d-2d-3 \le d-1$, we have $h^1(\mathcal {I}_{(A\cup B)\setminus (A\cup B)\cap H_2}(d-2)) =0$. 
Lemma \ref{w1} gives $A\setminus (A\cap (D\cup L)) = B\setminus (B\cap (D\cup L))$. 
Notice that either $\sharp (A\cap L) \ge (d+2)/2$, or $\sharp (B\cap L) \ge  
(d+2)/2$, since $\sharp ((A\cup B)\cap (D\cup L)) \ge 2d+3$ and $\sharp (A\cap (D\cup L)) 
= \sharp (B\cap (D\cup L))$.

Assume $x:= \sharp (A\cap L) \ge (d+2)/2$. Since $P\in \langle \nu _d(A)\rangle$
and $P\notin \langle \nu _d(A')\rangle$ for any $A'\subsetneq A$, the set $\langle \{P\} \cup 
\nu _d(A\setminus A\cap L)\rangle \cap \langle \nu _d(A\cap L)\rangle$ is a single point. 
Call $P_{L,A}$ this point. Since $A$ computes $sr(P)$, we see that 
$A\cap L$ computes the rank of $P_{L,A}$, with respect to the rational normal curve $\nu_d(L)$.  
Since $2x+1> d$, as explained in the proof of  Lemma \ref{w1}, $A\cap L$ is not an isolated 
point of $\mathcal{S}(P_{L,A})$ (w.r.t. $\nu_d(L)$). On the other hand, as in Claim 1,
adding $A\setminus (A\cap L)$ to a sets in $\mathcal {S}(P_{L,A})$
we obtain sets in $\mathcal{S}(P)$. As above, this implies that $A$ is not an isolated 
point of $\mathcal {S}(P)$. 

In the same way we conclude if $\sharp (B\cap D) \ge (d+2)/2$.

\quad (b) Here we assume the non-existence of a line $D\subset \mathbb {P}^m$ such that 
$\sharp (D\cap (A\cup B))\ge d+2$. Hence there is a conic $T\subset \mathbb {P}^m$ 
such that $\sharp (T\cap (A\cup B))\ge 2d+2$. 

Since $A$ computes $sr (P)$, the set 
$\langle \{P\}\cup \nu _d(A\setminus A\cap T)\rangle \cap \langle \nu _d(A\cap T)\rangle$ 
is a single point. Call this point $P_T$. Let $H_2$ be a general element of 
$\vert \mathcal {I}_T(2)\vert$. Since $\mathcal {I}_T(2)$ is spanned outside $T$ and 
$A\cup B$ is finite, we have $H_2\cap (A\cup B) = T\cap (A\cup B)$. Since $\sharp (A\cup B)
-\sharp ((A\cup B)\cap T) \le d-2\le d-1$, we have $h^1(\mathcal {I}_{A\cup B\setminus (A\cup B)
\cap H_2}(d-2))=0$. Lemma \ref{w1} gives $A\setminus A\cap T = B\setminus B\cap T$.

First assume that $T$ is a smooth conic. Hence $\nu _d(T)$ is a rational normal curve
of degree $2d$. In this case, the conclusion follows
by repeating the proof of the case 
$h^1(\mathcal {I}_{(A\cup B)\setminus (A\cup B)
\cap D}(d-1))=0$ of step (a), including Claim 1, with $\nu _d(T)$
instead of $\nu _d(D)$, and applying
Lemma  \ref{w1} for the integer $2d$. 

Now assume that $T$ is singular. Since $A\cup B$ 
is reduced, we may find $T$ as above which is not a double
line, say $T = L_1\cup L_2$ with $L_1\ne L_2$. Since $\sharp ((A\cup B)\cap T)\ge 2d+2$ 
and $\sharp ((A\cup B)\cap R)\le d+1$ for every line $R$, we have $\sharp ((A\cup B)\cap 
L_1)=\sharp ((A\cup B)\cap L_2) =d+1$ and $L_1\cap L_2\notin (A\cup B)$. If either
$\sharp (A\cap L_i) = \geq (d+2)/2$ or $\sharp (B\cap L_i) \geq (d+1)/2$ for some $i$, we may repeat 
the proof of the case $h^1(\mathcal {I}_{(A\cup B)\setminus (A\cup B)\cap D}(d-1))>0$ taking 
$L_1\cup L_2$ instead of $L\cup D$.

Thus, it remains to consider the case where $d$ is odd and $\sharp (A\cap L_i) = \sharp (B
\cap L_i) =(d+1)/2$ for all $i$. 
Set $\{O\}:= L_1\cap L_2$. Since $\langle \nu _d(L_1)
\rangle \cap \langle \nu _d(L_2)\rangle = \{\nu _d(O)\}$
and $P\notin \langle \nu _d(L_i)\rangle$, $i=1,2$, the linear space $\langle \nu _d(L_i)\rangle
\cap \langle \{\nu _d(P_T)\}\cup \nu _d(L_{2-i})\rangle$ is a line $D_i \subset 
\langle \nu _d(L_i)\rangle$ passing through $\nu _d(O)$. The set $\langle \nu _d(A\cap L_i)
\rangle \cap D_i$ is a point $P_{A,i}\in D_i\setminus\{\nu _d(O)\}$. Notice
that $\langle D_1\cup D_2\rangle$ is a plane and $P_T\in \langle D_1\cup D_2\rangle 
\setminus (D_1\cup D_2)$. Hence for each $U_1\in D_1\setminus \{\nu _d(O)\}$ there is a 
unique $U_2\in D_2\setminus \{O\}$ such that $P_T\in \langle \{U_1,U_2\}\rangle$. 
By construction, $P_{L_i,A}$ has symmetric tensor 
rank $sr _{L_i}(P_{L_i,A}) =(d+1)/2$ with respect to the rational normal
curve $\nu _d(L_i)$ (\cite{lt}, Theorem 4.1 or \cite{bgi}, \S 3) (we also have 
$sr (P) = (d+1)/2$, by \cite{ls}, Proposition 3.1). The non-empty
open subset $\langle \nu _d(L_i)\rangle \setminus \Sigma _{(d-1)/2}(\nu _d(L_i))$ of 
$\langle \nu _d(L_i)\rangle$ is the set of all $Q\in \langle \nu _d(L_i)\rangle$
whose symmetric rank with respect to $v_d(L_i)$ is exactly $sr _{L_i}(Q) = (d+1)/2$. 
Since $h^1(\mathbb {P}^1,\mathcal {I}_E(d)) =0$ for
every set $E\subset \mathbb {P}^1$ such that $\sharp (E)\le d+1$, for every
$Q\in \langle \nu _d(L_i)\rangle \setminus \Sigma _{(d-1)/2}(\nu _d(L_i))$ there is a 
unique $A_{i,Q} \subset L_i$ such that $\nu _d(A_{i,Q})$ computes $sr _{L_i}(P)$.
Set $\mathcal {U}_i:= \langle \nu _d(L_i)\rangle \setminus \Sigma _{(d-1)/2}(\nu _d(L_i)) 
\cap D_i$. For each $Q_1\in D_1\cap \nu _d(L_i)\rangle \setminus \Sigma _{(d-1)/2}(\nu _d(L_i)$,
call $Q_2$ the only point
of $D_2\setminus \{O\}$ such that $P\in \langle \{Q_1,Q_2\}\rangle$. 
By moving $Q_1 in D_1$, we find an integral one-dimensional
variety $\Delta := \{F\cup A_{L_1,Q_1}\cup A_{L_2,Q_2}\} \subseteq \mathcal {S}(P)$ 
with $A\in \Delta$. Hence $A$ is not an isolated point of 
$\mathcal {S}(P)$.\qed

\medskip

The following example shows the bound $sr (P) < 3d/2$ 
in the statement of Theorem \ref{i2} is sharp, for large $d$.

\begin{example}\label{i1}
Fix an even integer $d \ge 6$. Assume $m \ge 2$. Here we construct 
$P\in \mathbb {P}^n$ such that $sr (P) =3d/2$ and its symmetric rank is computed
by exactly two subsets of $X_{m,d}$ 

Fix a $2$-dimensional linear subspace $M\subseteq \mathbb {P}^r$ 
and a smooth plane cubic $C\subset M$. Since $h^1(M,\mathcal {I}_C(d)) = 
h^1(M,\mathcal {O}_M(d-3)) =0$, we have $\deg (\nu _d(C)) =3d$, $\dim (\langle \nu _d(C)\rangle ) 
=3d-1$ and $\nu _d(C)$ is a linearly normal elliptic curve of $\langle \nu _d(C)\rangle$. 
Since no non-degenerate curve is defective (\cite{a}, Remark 1.6), we have $\Sigma _{3d/2}(\nu _d(C))
=\langle \nu _d(C)\rangle$ and $\Sigma _{3d/2}(\nu _d(C))\setminus \Sigma _{(3d-2)/2}(\nu _d(C))$ 
is a non-empty open subset of the secant variety $\Sigma _{3d/2}(\nu _d(C))$. 
Fix a general $P\in \Sigma _{3d/2}(\nu _d(C))$. 
Since $\nu _d(C)$ is not a rational normal curve, by \cite{cc1}, Theorem 3.1 and \cite{cc1}, Proposition 5.2, there are exactly $2$ (reduced) subsets of $\nu_d(C)$, 
of cardinality $3d/2$, which compute the symmetric
rank of $P$. Thus, to settle the example, it is sufficient to prove that 
any $B\subset \mathbb {P}^m$  such that $\nu _d(B)$ computes $sr (P)$, is 
a subset of $C$. Obviously $\sharp (B) \le 3d/2$.

Assume $B\nsubseteq C$. Let $H_3$ be a general cubic hypersurface containing
$C$ (hence $H_3 = C$ if $r=2$). Set $B':= B\setminus B\cap C$. Since $B$ is finite and $H_3$ is general, 
we have $B\cap H_3 = B\cap C$. Since $A\subset C$, we have $B' = (A\cup B)\setminus (A\cup B)\cap C$. 
Lemma \ref{v1} gives $h^1(\mathcal {I}_{A\cup B}(d)) >0$. Hence $h^1(M,\mathcal {I}_{A\cup B}(d)) > 0$. 
Remark \ref{a00} gives that either $h^1(C,\mathcal {I}_{(A\cup B)\cap C}(d)) > 0$ 
or $h^1(\mathcal {I}_{B'}(d-3)) > 0$.

\quad (a) First assume $h^1(\mathcal {I}_{B'}(d-3)) > 0$. Since $d\ge 3$ and $\sharp (B') \le 2d-1$, 
there is a line $D\subset M$ such that $\sharp (D\cap B') \ge d-1$
(see \cite{bgi}, Lemma 34, or \cite{c2}, Th. 3.8). Since $\nu _d(B)$
is linearly independent, we have $\sharp (D\cap B)\le d+1$. 

Assume $\sharp (D\cap (A\cup B)) \le d+1$. Hence $h^1(D,\mathcal {I}_{(A\cup B)\cap D}(d))=0$. 
Remark \ref{a00} gives $h^1(M,\mathcal {I}_{(A\cup B)\setminus (A\cup B)\cap D}(d-1)) >0$. 
Set $F:= (A\cup B)\setminus ((A\cup B)\cap D)$. We easily compute $\sharp (F) < 3(d-1)$. 
By \cite{c2}, Theorem 3.8, we get that either there is a line $D_1$ such that $\sharp (F\cap D_1)\ge d+1$ 
or there is a conic $D_2 $ such that $\sharp (D_2\cap F) \ge 2d$.  
As $P\in \Sigma _{3d/2}(\nu _d(C))$ is general, then also $A$ is general in $C$ (hence reduced). 
Thus, no $3$ of its points are collinear  and no $6$ of its points are contained in a conic.
Hence if $D_1$ exists, we get $\sharp (B) \ge 2d-2$, while if $D_2$ exists, we get $\sharp (B)  
\ge d-1 +(2d-5) =3d-6$; both lead to a contradiction, because $d\ge 6$ and $\sharp (B) =3d/2$. 

Now assume $\sharp (D\cap (A\cup B)) \ge d+2$. Let $H\subset \mathbb {P}^m$ be a general
hyperplane containing $D$. Since $A\cup B$ is finite and $H$ is general, we have 
$H\cap (A\cup B) = D\cap (A\cup B)$. If $h^1(\mathcal {I}_{(A\cup B)\setminus (A\cup B)\cap H}(d-1)) =0$, 
then Lemma \ref{v2} gives $B\setminus B\cap D = A\setminus A\cap D$. Hence $\sharp (A\cap D) = 
\sharp (B\cap D)$. Since $\sharp (A\cap D) \le 2$, we get $d\le 2$, a contradiction.
Now assume $h^1(\mathcal {I}_{(A\cup B)\setminus (A\cup B)\cap H}(d-1))>0$. Since 
$\sharp ((A\cup B)\setminus (A\cup B)\cap H) \le  2d-2$, there is a line $L\subset \mathbb {P}^m$
such that $\sharp (L\cap (A\cup B)\setminus ((A\cup B)\cap D)) \ge d+1$. Let $H_2\subset \mathbb {P}^m$ 
be a general quadric hypersurface containing $L\cup D$. As usual, since $A\cup B$ is finite, 
$L\cup D$ is the base locus of the linear system $\vert \mathcal {I}_{L\cup D}(2)\vert$ 
and $H_2$ is general in  $\vert \mathcal {I}_{L\cup D}(2)\vert$, we have 
$H_2\cap (A\cup B) = (L\cup D)\cap (A\cup B)$. 
Since $\sharp ((A\cup B)\setminus (A\cup B)\cap H_2) \le d-3$,
we have $h^1(\mathcal {I}_{(A\cup B)\setminus (A\cup B)\cap H_2}(d-2)) =0$. Hence 
Lemma \ref{v2} gives $A\setminus A\cap H = B\setminus B\cap H$. Hence $\sharp ((A\cap (L\cup D)) = 
\sharp (B\cap (L\cup D))$. This is absurd, because $d\ge 4$ while, by generality, 
no $6$ points of $A$ are on a conic.

\quad (b) Assume $h^1(C,\mathcal {I}_{(A\cup B)\cap C}(d)) > 0$ and $h^1(\mathcal {I}_{B'}(d-3))=0$. 
Since $C$ is a smooth elliptic curve and $\deg (\mathcal {O}_C(d)) =3d$,
either $\deg ((A\cup B)\cap C) \ge 3d+1$ or $\deg ((A\cup B)\cap C) = 3d$ and 
$\mathcal {O}_C((A\cup B)\cap C) \cong \mathcal {O}_C(d)$. Hence $\sharp (B\cap C) \ge (3d-1)/2$. 
Therefore $\sharp (B') \le 2$. Taking $D:= C$ in Lemma \ref{v2} we get $B'=\emptyset$, 
because $A\subset C$.
\end{example}

\vspace{0.3cm}

Next, we prove Theorem \ref{i3}, a more precise description of the positive dimensional
components of $\mathcal{S}(P)$,  when $sr(P)<3d/2$. 

\qquad {\emph {Proof of Theorem \ref{i3}.}} Fix $A\in \mathcal {S}(P)$. 
and assume the existence of $B\in \mathcal {S}(P)$ such that $B\ne A$. 
At the beginning of the proof of Theorem \ref{i2}
we showed that either:
\begin{itemize}
\item[(i)] there is a line $D\subset \mathbb {P}^r$ such that $\sharp (D\cap (A\cup B))\ge d+2$;
\item[(ii)] there is a conic $T\subset \mathbb {P}^r$ such that $\sharp (T\cap (A\cup B))\ge 2d+2$.
\end{itemize}

\quad (i) Here we assume the existence of a line $D\subset  \mathbb {P}^r$ such that 
$\sharp ((A\cup B)\cap D) \ge d+2$. We proved in step (a) of the proof of Theorem \ref{i2} 
that $\sharp (A\cap D) = \sharp (B\cap D)$. Hence $\sharp (A\cap D) \ge \lceil (d+2)/2\rceil$.
Set $F:= A\setminus A\cap D$. Since $P\in \langle \nu _d(A)\rangle$ and 
$P\notin \langle \nu _d(A')\rangle$ for any $A'\subsetneq A$, the set $\langle 
\nu _d(A\cap D)\rangle \cap \langle \{P\}\cup \nu _d(F)\rangle$ is a single point. Let $P_D$ denote
this point. Lemma \ref{w1} and the symmetric case of \cite{bl}, Corollary 2, give that
$\mathcal {S}(P_D)$ is infinite and each element of it is contained in $D$. 
Thus, to prove that we are in case (a) of the statement, it is sufficient to prove that 
$E\cup F\in \mathcal {S}(P)$ for any $E\in \mathcal {S}(P_D)$. 
This assertion is just Claim 1 of the proof of Theorem \ref{i2}.

\quad (ii) Now assume the non-existence of a line $D$ as above. Then, there is a (reduced) conic
$T\subset  \mathbb {P}^r$ such that $\sharp (T\cap (A\cup B)) \ge 2d+2$ and 
$A\setminus A\cap T = B\setminus B\cap T$.
Hence $\sharp (A\cap T) = \sharp (B\cap T) \ge d+1$. We consider separately the cases in which
$T$ is smooth or $T$ is singular.

\qquad (ii.1) Assume $T$ is smooth. Set $F:= A\setminus A\cap T$. As in step (i), 
we see that $\langle \nu _d(A\cap D)\rangle \cap \langle \{P\}\cup \nu _d(F)\rangle$ is a single point,
$P_T$. Moreover, we see that $\sharp (A\cap T) = sr (P_T)$ and $\mathcal {S}(P_T)$ is infinite,
since $\{F\cup E\}_{E\in \mathcal {S}(P_T)}\subseteq \mathcal {S}(P)$. To conclude that we are in case (b),
we need to prove that every element of $\mathcal {S}(P)$ is of the form $F\cup E$, 
$E\in \mathcal {S}(P_T)$. Fix any $B\in \mathcal {S}(P)$ such that $B \ne A$. Since $\sharp (A\cup B)<3d$ 
and $h^1(\mathcal {I}_{A\cup B}(d)) >0$, either there is a line $D_1$ such that 
$\sharp ((A\cup B)\cap D_1) \ge d+2$, or there is a reduced conic $T_2\neq T$ such that 
$\sharp ((A\cup B)\cap T_2) \ge 2d+2$ (\cite{c2}, Theorem 3.8). 

Assume the existence of the line $D_1$. If $h^1(\mathcal {I}_{(A\cup B)\setminus 
((A\cup B)\cap D_1)}(d-1)) =0$, then Lemma \ref{v2} gives $A\setminus
A\cap D_1 = B\setminus B\cap D_1$. Since
$\sharp (A) = sr (P) = \sharp (B)$, we get $\sharp (A\cap D_1)=\sharp (B\cap D_1)
\geq (d+2)/2$, which contradicts the fact that we are not in case (i).
Therefore $h^1(\mathcal {I}_{(A\cup B)\setminus
(A\cup B)\cap D_1}(d-1))>0$. Hence there is a line $D_2$ such that $\sharp (D_2\cap ((A\cup B)\setminus
(A\cup B)\cap D_1))\ge d+1$. Let $H_2$ be a general quadric hypersurface containing $D_1\cup D_2$ 
(it exists, because if $D_1\cap D_2 =\emptyset$, then $m\ge 3$). 
Since $\sharp ((A\cup B)\setminus (A\cup B)\cap H_2)\le (3d-1)-2d-3
\le d-1$, we have $h^1(\mathcal {I}_{(A\cup B)\setminus (A\cup B)\cap H_2}(d-2))=0$. 
Hence Lemma \ref{v2} implies $A\setminus A\cap H_2 = B\setminus B\cap H_2$. Since 
$H_2$ be a general quadric hypersurface containing $D_1\cup D_2$, we have 
$A\cap H_2 = A\cap (D_1\cup D_2)$ and $B\cap H_2 = B\cap (D_1\cup D_2)$. Since
$T\cap (D_1\cup D_2) \le 4$, we get $2d+3 \le \sharp ((A\cup B)\cap (D_1\cup D_2)) \le 8$, 
contradicting the assumption $d \ge 3$.

Assume the existence of the conic $T_2$ and assume $T\neq T_2$. In step (ii) of the proof of Theorem \ref{i2}, 
we proved that $A\setminus T_2\cap A = B\setminus T_2\cap B$. Since $\sharp (A) = sr (P) =\sharp (B)$, 
we get $\sharp (A\cap T_2) = \sharp (B\cap T_2)$. Since $\sharp (T\cap
T_2) \le 4$ and $\sharp (A\setminus A\cap T) \le (3d-1)/2-d-1$, we have 
$\sharp (A\cap T_2) \le (3d-1)/2 -d+3 = (d+5)/2$. Hence $\sharp (A\cap T_2) = 
\sharp (B\cap T_2) \ge 2d+2 -(d+5)/2 = (3d-1)/2$. Since $\sharp (A\cap T_2) +\sharp (B\cap T_2) \ge 
\sharp ((A\cup B)\cap T_2) \ge 2d+2$ we get $d=3$ and $A\subset T$. Hence $\sharp(B\cap T_2)\geq 4$ 
so that $B\subset T_2$. Thus $A\subset T $ and $B\subset T_2$ and moreover 
$A\setminus A\cap T_2 = B\setminus B\cap T_2=\emptyset$. It follows that $A = T\cap T_2$. 
Since $A\subset T$ and $T$ is a smooth conic, we have $P\in \langle \nu _3(T)\rangle$
and the symmetric rank of $P$, with respect to the rational normal curve $\nu _3(T)\subset\mathbb {P}^6$,
is $4$. It follows that $\mathcal {S}(P)$ is infinite. By the symmetric case
of \cite{bl}, Corollary 2.2, we have $B\subset \nu _3(T)$ for all $B\in \mathcal {S}(P)$. 
Hence (b) holds, in this case.

Finally, assume that $T_2$ exists and $T=T_2$. I.e. assume $\sharp (T\cap (A\cup B)) \ge 2d+2$. 
In step (ii) of the proof of Theorem \ref{i2}, we proved that
$A\setminus T\cap A = B\setminus T\cap B$ and that $B\cap T$ computes $sr (P_T)$. 
Hence $B\in \{F\cup E\}_{E\in \mathcal {S}(P_T)}$.

\qquad (ii.2) Here we assume the existence of a {\it reducible} conic $T$ such that 
$\sharp (A\cap T)\ge d+1$. Write $T = L_1\cup L_2$ with $\sharp (A\cap L_1) \ge \sharp (A\cap L_2)$. 
If $\sharp (A\cap L_1)\ge (d+2)/2$, then, by step (i), we are in case (a). 
If $\sharp (A\cap L_1)< (d+2)/2$, then we get $\sharp (A\cap L_1) = \sharp (A\cap
L_2)=(d+1)/2$ and $L_1\cap L_2\notin A$. We also get that $d$ is odd. 
It remains simply prove that $\mathcal {S}(P) \ne \{A\}$. Indeed, we proved
that $\mathcal {S}(P)$ is infinite in the second part of step (ii) of the proof of Theorem \ref{i2}.

The proof of the statement is completed.\qed

\providecommand{\bysame}{\leavevmode\hbox to3em{\hrulefill}\thinspace}


\begin{thebibliography}{99}

\bibitem{a}  B. \r{A}dlandsvik, Joins and higher secant varieties. Math.
Scand. {\bf {61}} (1987), 213--222.


\bibitem{bb} E. Ballico and A. Bernardi, Decomposition of homogeneous polynomials with low rank,
arXiv:1003.5157v2~[math.AG], Math. Z. .DOI :10.1007/s00209-011-0907-6

\bibitem{bb1} E. Ballico and A. Bernardi, A partial stratification of secant varieties of 
Veronese varieties via curvilinear subschemes, arXiv:1010.3546v2 ~[math.AG].

\bibitem{bc} E. Ballico and L. Chiantini, A criterion for detecting the identifiability
of symmetric tensors of size three, arXiv:1202.1741~[math.AG].

\bibitem{bgi} A. Bernardi, A., Gimigliano and M. Id\`{a}, Computing symmetric rank for 
symmetric tensors. J. Symbolic. Comput. \textbf{46} (2011), no. 1, 34--53.

\bibitem{bcmt} J. Brachat, P. Comon, B. Mourrain and E. P. Tsigaridas, 
Symmetric tensor decomposition.
Linear Algebra Appl. {\bf {433}} (2010), no. 11--12, 1851--1872.


\bibitem{bb+} W. Buczy\'{n}ska, J. Buczy\'{n}ski. Secant varieties to high degree Veronese 
reembeddings, catalecticant matrices and smoothable Gorenstein schemes. arXiv:1012.3562v4~[math.AG].

\bibitem{bgl} J. Buczy\'{n}ski, A. Ginensky and J. M. Landsberg,
Determinantal equations for secant varieties and the 
Eisenbud-Koh-Stillman conjecture, arXiv:1007.0192v4~[math.AG].

\bibitem{bl} J. Buczy\'{n}ski and J. M. Landsberg, Rank of tensors and a generalization 
of secant varieties, arXiv:0909.4262v3~[math.AG].

\bibitem{cc1} L. Chiantini and C. Ciliberto, On the concept of k-secant order of a variety, 
J. London Math. Soc. (2) {\bf {73}} (2006), no. 2, 436--454.

\bibitem{cs} G. Comas and M. Seiguer, On the rank of a binary form,  
Found. Comp. Math. \textbf{11} (2011), no. 1, 65--78.

\bibitem{cglm} P. Comon, G. H. Golub, L.-H. Lim and B. Mourrain,
Symmetric tensors and symmetric tensor rank,
SIAM J.  Matrix Anal.  {\bf {30}} (2008) 1254--1279.

\bibitem{c2} A. Couvreur, The dual minimum distance of arbitrary
dimensional algebraic-geometric codes, J. Algebra {\bf {350}} (2012), no.
1, 84--107.

\bibitem{h} R. Hartshorne, Algebraic Geometry. Springer-Verlag, Berlin, 1977.

\bibitem{kb} T. Kolda and B. Bader, Tensor decompositions and applications, SIAM
Review {\bf 51} (2009), 455-500. 

\bibitem{l} J. M. Landsberg, Tensors: Geometry and Applications
Graduate Studies in Mathematics, Vol. 118, Amer. Math. Soc. Providence, December 2011.

\bibitem{lt} J. M. Landsberg  and Z. Teitler,
On the ranks and border ranks of symmetric tensors.
Found. Comput. Math. \textbf{10} (2010), no. 3, 339--366.

\bibitem{ls} L.-H. Lim and V. de Silva,  Tensor rank and the ill-posedness of the best 
low-rank approximation problem, SIAM J. Matrix Anal. Appl. \textbf{30} (2008), no. 3,1084--1127.




\end{thebibliography}
\end{document}